\documentclass[11pt]{amsart}

\usepackage{fullpage}
\usepackage{amssymb}
\usepackage{amsmath}
\usepackage{amsxtra}
\usepackage{amscd}
\usepackage{graphicx}
\usepackage{amsfonts}
\usepackage{pb-diagram}
\usepackage{float}
\usepackage{enumerate}
\usepackage{dsfont}
\usepackage{color}
\usepackage{url}
\usepackage{hyperref}
\definecolor{darkgreen}{rgb}{0,0.4,0.1}
\hypersetup{
   colorlinks=true,       	
   linkcolor=red,          
   citecolor=darkgreen,    
   filecolor=magenta,      
   urlcolor=blue			
}

\def\a{\alpha}
\def\b{\beta}

\def\d{\delta}

\def\s{\sigma}
\def\leq{\leqslant}
\def\geq{\geqslant}

\numberwithin{equation}{section}
 \newtheorem{nono-thm}{Theorem}
 \newtheorem{thm}[equation]{Theorem}
\newtheorem{lemma}[equation]{Lemma}

\newtheorem{prop}[equation]{Proposition}

\newtheorem{rem}[equation]{Remark}
\newtheorem{note}[equation]{Note}

\def\wt{\widetilde}
\def\Z{{\mathbb Z}}
\def\R{{\mathbb R}}
\def\C{{\mathbb C}}
\def\N{{\mathbb N}}
\def\Jtr{{\rm tr}_d}
\def\YH{{\rm Y}_{d,n}}

\def\Jtrs{{\rm tr}_{d,D}}

\def\zet{z}

\begin{document}

\title[On the link invariants from the Yokonuma--Hecke algebras]
  {On the link invariants from the Yokonuma--Hecke algebras}

\author{S. Chmutov}
\address{Ohio State University,
1680 University Drive, Mansfield, OH 44906. \\
USA.}
\email{chmutov@math.ohio-state.edu}
\urladdr{https://people.math.osu.edu/chmutov.1}

\author{S. Jablan$^\dagger$}
\address{$^\dagger$(1952 - 2015) The Mathematical Institute, Knez Mihailova 36, B.O. Box 367, 11001 Belgrade, Serbia.}

\author{K. Karvounis}
\address{Institut f\"{u}r Mathematik,
Universit\"{a}t Z\"{urich},
Winterthurerstrasse 190, CH-8057 Z\"{u}rich, Switzerland.}
\email{konstantinos.karvounis@math.uzh.ch}

\author{S. Lambropoulou}
\address{Departament of Mathematics,
National Technical University of Athens,
Zografou campus, GR-157 80 Athens, Greece.}
\email{sofia@math.ntua.gr}
\urladdr{http://www.math.ntua.gr/$\wt{\ }$sofia}

\thanks{This research  has been co-financed by the European Union (European Social Fund - ESF) and Greek national funds through the Operational Program ``Education and Lifelong Learning" of the National Strategic Reference Framework (NSRF) - Research Funding Program: THALES: Reinforcement of the interdisciplinary and/or inter-institutional research and innovation.}

\keywords{classical braids, framed braids, Yokonuma--Hecke algebras, Markov trace, framed knots and links, E--condition, classical knots and links, transverse knots and links, singular braid monoid, singular knots and links, Homflypt polynomial.}

\subjclass[2010]{57M27, 57M25, 20F36, 20F38, 20C08}
\date{}

\begin{abstract}
In this paper we study properties of the Markov trace $\Jtr$ and the specialized trace $\Jtrs$ on the Yokonuma--Hecke algebras, such as behaviour under inversion of a word, connected sums and mirror imaging. We then define invariants for framed, classical and singular links through the trace $\Jtrs$ and also invariants for transverse links through the trace $\Jtr$. In order to compare the invariants for classical links with the Homflypt polynomial we develop computer programs and we evaluate them on several Homflypt-equivalent pairs of knots and links. Our computations lead to the result that these invariants are topologically equivalent to the Homflypt polynomial on \textit{knots}. However, they do not demonstrate the same behaviour on links.
\end{abstract}

\maketitle

\section*{Introduction}

In his pioneering work \cite{jo}, V.F.R. Jones constructed the Homflypt polynomial $P(q,\zet)$, an isotopy invariant of classical knots and links, using the Iwahori--Hecke algebras ${\rm H}_n(q)$, the Ocneanu trace $\tau$ and the natural surjection of the classical braid groups $B_n$ onto the algebras ${\rm H}_n(q)$. In \cite{jula2} the Yokonuma--Hecke algebras have been used for constructing framed knot and link invariants following the method of Jones. 

\smallbreak
 The Yokonuma--Hecke algebra $\YH(q)$ is a quotient of the framed braid group ${\mathcal F}_n$ over a modular relation (Eq.~\ref{modular})  for the framing generators $t_j$  and a quadratic relation (Eq.~\ref{quadr}) for the braiding generators $g_i$, which involves certain idempotents $e_i$ (Eq.~\ref{ei}). We note that for $d=1$ the algebra ${\rm Y}_{1,n}(q)$ coincides with the algebra ${\rm H}_n(q)$. An appropriate inductive basis for $\YH(q)$ and a Markov trace $\Jtr$ constructed on $\cup_n\YH(q)$  (\cite{ju}, see Theorem~\ref{trace}) pointed at the construction of framed link invariants from these algebras. The trace $\Jtr$ depends on parameter $z$ for the Markov property and parameters $x_1, \ldots, x_{d-1}$ for counting the framings.  As it turned out, the trace $\Jtr$ would re-scale to also respect negative stabilization only upon imposing the {\rm E}--condition, which is equivalent to a non-linear system of equations involving the trace parameters $x_1, \ldots, x_{d-1}$  (Eq.~\ref{Esystem}) \cite{jula2}. The complete set of solutions of the {\rm E}--system  was given by P.~G\'{e}rardin (cf. Appendix of \cite{jula2}) and he showed that they are parametrized by the non-empty subsets of $\Z/d\Z$. Then, specializing the trace parameters $x_1, \ldots, x_{d-1}$ to a solution of the {\rm E}--system $({\rm x}_1, \ldots , {\rm x}_{d-1})$ parametrized by the the subset $D$ of $\Z/d\Z$, we obtain the specialized trace $\Jtrs$ \cite{chla}. Consequently, J.~Juyumaya and S.~Lambropoulou defined 2-variable invariants for framed knots and links \cite{jula2}. Further, since the classical braid group $B_n$ embeds in the framed braid group  ${\mathcal F}_n$, classical links can be viewed as framed links with all framings zero. Thus, through the construction above we also obtain invariants for classical links \cite{jula3}. Further, via an appropriate homomorphism from the singular braid monoid to the Yokonuma--Hecke algebras, one can also define invariants for singular knots and links \cite{jula4}.
 
In this paper we adapt the constructions in \cite{jula2,jula3,jula4} for defining invariants for framed, classical and singular knots and links respectively, which are denoted by $\Phi_{d,D}$, $\Theta_{d,D}$ and $\Psi_{d,D}$ respectively, using a different quadratic relation for the Yokonuma--Hecke algebra \cite{chpa}. Moreover, we construct invariants for transverse knots and links, for which no use of the {\rm E}--condition is needed. 
Indeed, as we observe, the braid equivalence in the class of transverse knots and links requires only positive stabilization \cite{OrSh,Wr}, making the algebras $\YH(q)$ a natural algebraic companion. This leads to the construction of a ($d+1$)-variable transverse link invariant, 
$M_d(q, z, x_1, \ldots , x_{d-1})$, via the algebras $\YH(q)$ and the trace $\Jtr$ (see Theorem~\ref{transverseinv}).

\smallbreak
We focus now on the class of classical knots and links. The invariants of classical knots and links defined via the Yokonuma--Hecke algebras need to be compared with other known invariants, especially with the Homflypt polynomial  $P(q,z)$.  The first result in this direction was obtained in \cite{chla} where it was shown that these invariants coincide with the polynomial  $P(q,z)$ only when we are in the group algebra ($q=\pm 1$) or when $\Jtr(e_i) = 1$, which is equivalent to the solution of the {\rm E}--system comprising the $d$-th roots of unity. Further, it was shown in \cite{chla}  that there is no algebra homomorphism between the Yokonuma--Hecke algebra and the Iwahori--Hecke algebras which respects the trace rules, except when $\Jtr(e_i) = 1$. Moreover, it is very difficult to make a comparison of the invariants diagrammatically, since the skein relation of the framed link invariants \cite[Proposition~7]{jula2} has no topological meaning (see Figure~\ref{framed_skein}) for the classical link invariants \cite{jula3}.  
 Yet, the classical link invariants may be topologically equivalent with the polynomial $P$, in the sense that they do not distinguish more or less knot and link pairs. 
 Computations were now required, but their complexity would increase drastically due to the appearance of the $e_i$'s in the quadratic relation.
Consequently, K.~Karvounis and M.~Chmutov developed independently computational packages, which were cross-checked. The invariants $\Theta_{d,D}$ were computed on several pairs of distinct knots sharing the same Homflypt polynomial, for small values of $d$. These computations indicated that these invariants also do not distinguish those pairs of knots. They also led S.~Jablan and K.~Karvounis to formulate conjectures about the comparison of the invariants  $\Theta_{d,D}$ and the Homflypt polynomial on knots and also on their behaviour under mirror imaging (see Proposition~\ref{mirror}). Namely, the following Conjecture, which is now a Theorem proved in the sequel paper \cite{chjukala} and which comprises our main result in this paper, provides a direct derivation of the polynomial $P$ from the invariants $\Theta_{d,D}$ for the case of \textit{knots}:

\begin{nono-thm}[Comparison with the Homflypt polynomial]\label{thm2}
Given a solution $X_D$ of the {\rm E}--system, for any braid $\a \in B_{n}$ such that $\widehat{\a}$ is a knot, we have that:
\begin{center}
$\Theta_{d,D} (\widehat{\a})(q,z) = \Theta_{1,\{0\}} (\widehat{\a})(q,d\,z) = P(\widehat{\a})(q,d\,z)$
\end{center}
or equivalently with a change of variables:
\begin{center}
$\Theta_{d,D} (\widehat{\a})(q,\lambda_D) =\Theta_{1,\{0\}} (\widehat{\a})(q,\lambda_D)$.
\end{center}
\end{nono-thm}

Surprisingly, Theorem~\ref{thm2} does not hold for \textit{links}, as we demonstrate in \S\ref{links}. In \S\ref{conjclassic} and in the Appendix we  present lists of links and of knots respectively, sharing the same Homflypt polynomial. These were obtained via the program {\it LinKnot} \cite{jasa}, which we used for formulating our Conjecture.

\smallbreak
The paper is organized as follows. In \S\ref{yh} we recall the definition and basic relations in the Yokonuma--Hecke algebras $\YH(q)$. We use the new quadratic relation \cite{chpa}, which is more economical with respect to computations, and we adjust all equations needed. 
In \S\ref{traces} we recall the Ocneanu trace $\tau$ on the algebras ${\rm H}_n(q)$, and also the trace $\Jtr$ and the specialized trace $\Jtrs$ on the algebras $\YH(q)$. The passage from $\Jtr$ to $\Jtrs$ is via the E--system, which is also presented. In \S\ref{properties} we prove some important properties of the traces, which are known for the trace $\tau$. Namely, invariance under inversion of a braid word (satisfied by $\Jtr$, hence also by $\Jtrs$ and $\tau$), multiplicativity under the connected sum operation of $\Jtrs$  (hence also of $\tau$ but not of $\Jtr$), behaviour on split links of $\Jtr$ (hence also of $\Jtrs$ and $\tau$), and the mirroring effect of $\Jtrs$ under the mirror imaging of a braid word (hence also of $\tau$ but not of $\Jtr$). 
We continue with \S\ref{yhinvts} where we present the invariants from the Yokonuma--Hecke algebras for different classes of knots and links: framed, classical, singular and transverse, and where the formuli are adapted to the new quadratic relation. In \S\ref{trnsv} we give the construction of the transverse link invariant $M_d$. For each invariant we discuss which trace properties it inherits. In \S\ref{comp} we present our computational algorithm. In \S\ref{results_transverse}, computations on well--known families of transverse knots led to the observation that the invariants $M_d$ do not give something new. Finally, in \S\ref{conjclassic}, we present the computational data which culminate to Theorem~\ref{thm2} and we discuss the behaviour of the invariants $\Theta_{d,D}$ on links.

In a further development, for the case of classical knots and links J.~Juyumaya has conjectured that the trace $\Jtrs$ can be evaluated only by rules involving the generators $g_i$ and the elements $e_i$, thus treating the elements $e_i$ as formal elements. This fact is proved in the sequel paper \cite{chjukala}. Consequently, a computer program has been developed using this result for speeding up the calculations (\cite{ka}, see also \url{http://www.math.ntua.gr/~sofia/yokonuma}). In \cite{chjukala} we also prove Theorem~\ref{thm2} and we provide a concrete formula for relating the invariants $\Theta_{d,D}$ on links with the Homflypt polynomial. Further, we show that the invariants $\Theta_{d,D}$ distinguish six pairs of Homflypt-equivalent links (one of them is the fourth pair in Table~\ref{link_table}).

\bigbreak
We would like to thank the Referee for some useful comments. Also, the last two authors acknowledge with pleasure discussions with Maria Chlouveraki.
\smallbreak

Finally, we would like to mention that the involvement of Slavik Jablan has been crucial in the advancement of the comparison of the invariants $\Theta_{d,D}$ with the Homflypt polynomial. More precisely, the results in \cite{chjukala} are based on Theorem~\ref{thm2}, which would not have been formulated without the contribution of Slavik Jablan.

\section{The Yokonuma--Hecke algebra}\label{yh}
In this section we recall the definition of the Yokonuma--Hecke algebra as a quotient of the framed braid group.

\subsection{The framed braid group and the modular framed braid group} The {\it framed braid group}, ${\mathcal F}_n \cong \Z^n\rtimes B_n$, is the group  defined by the standard generators  $\s_1, \ldots, \s_{n-1}$ of the classical braid group $B_n$ together with the framing  generators  $t_1, \ldots, t_n$ ($t_j$ indicates framing 1 on the $j$th strand),  subject to the relations:
\begin{equation}\label{framedgp}
\begin{array}{ccrclcl}
\mathrm{(b_1)}& & \s_i\s_j\s_i & = & \s_j\s_i\s_j && \mbox{for $ \vert i-j\vert = 1$}\\
\mathrm{(b_2)}& & \s_i\s_j & = & \s_j\s_i & & \mbox{for $\vert i-j\vert > 1$}\\
\mathrm{(f_1)}& & t_i t_j & =  &  t_j t_i &&  \mbox{for all $ i,j$}\\
\mathrm{(f_2)}& & t_j \s_i & = & \s_i t_{s_i(j)} && \mbox{for all $ i,j$}
\end{array}
\end{equation}
where $s_i(j)$ is the effect of the transposition $s_i:=(i, i+1)$ on $j$. Relations $\mathrm{(b_1)}$ and $\mathrm{(b_2)}$ are the usual braid relations, while relations $\mathrm{(f_1)}$ and $\mathrm{(f_2)}$ involve the framing generators. Further, for a natural number $d$ the {\it modular framed braid group}, denoted
${\mathcal F}_{d,n}$, can be defined as the group with the presentation of the framed braid group, but including also the modular relations:
\begin{equation}\label{modular}
\begin{array}{ccrclcl}
\mathrm{(m})& & t_j^d   & =  &  1 && \mbox{for all $j$}
\end{array}
\end{equation}
Hence, ${\mathcal F}_{d,n}\cong (\Z/d\Z)^n\rtimes B_n$. Geometrically, the elements of ${\mathcal F}_{n}$ (respectively ${\mathcal F}_{d,n}$) are classical braids on $n$ strands with an integer (respectively an integer modulo $d$), the framing, attached to each strand. Further, due to relations $\mathrm{(f_1)}$ and $\mathrm{(f_2)}$, every framed braid $\a$ in ${\mathcal F}_{d,n}$ can be written in its {\it split form} as $\a = t_1^{k_1} \ldots t_n^{k_n} \s$, where $k_1, \ldots, k_{n-1} \in \Z$ and $\s$ involves only the standard generators of $B_n$. The same holds also for the modular framed braid group.

For a fixed $d$ we define the following elements $e_i$ in the group algebra ${\C}{\mathcal F}_{d,n}$:
\begin{equation}\label{ei}
e_i := \frac{1}{d}\sum_{1\leq s\leq d}t^s_i t^{-s}_{i+1} \qquad (1\leq i\leq n-1)
\end{equation}
where $-s$ is considered modulo $d$.
One can easily check that $e_i$ is an idempotent: $e_i^2=e_i$  and that $e_i\s_i=\s_ie_i$ for all $i$.

\subsection{The Yokonuma--Hecke algebra}
Let $d\in \N$ and let $q \in \C \backslash \{0\}$  fixed. The {\it Yokonuma--Hecke algebra}  $q$, denoted $\YH(q)$, is defined as the quotient of
$\C {\mathcal F}_{d,n}$ by factoring through the ideal generated by the expressions:
$ \s_i^2 - 1 - (q - q^{-1})e_i \s_i$ for $1\leq i\leq n-1$.
We shall denote $g_i$ the element in the algebra $\YH(q)$ corresponding to $\s_i$ while we keep the same notation for $t_j$ in the algebra $\YH(q)$. So, in $\YH(q)$ we have the following quadratic relations:

\begin{equation}\label{quadr}
g_i^2 = 1 + (q - q^{-1}) \, e_i \, g_i \qquad (1\leq i\leq n-1).
\end{equation}
The elements $g_i\in\YH(q)$ are invertible:
\begin{equation}\label{invrs}
g_i^{-1} = g_i - (q - q^{-1}) \, e_i \qquad (1\leq i\leq n-1).
\end{equation}

Further the elements $g_i \in \YH(q)$ satisfy the following relations:

\begin{lemma}
Let $i \in \{1,\ldots,n-1\}$. Then:
\begin{align*}
g_i^r &= (1-e_i)\,g_i + \left( \frac{q^r+q^{-r}}{q+q^{-1}} \right)e_ig_i +\left(  \frac{q^{r-1}-q^{-r+1}}{q+q^{-1}}\right) e_i &\text{ for $r$ odd},\\
g_i^r &= 1-e_i + \left( \frac{q^r-q^{-r}}{q+q^{-1}} \right)e_ig_i +\left(  \frac{q^{r-1}+q^{-r+1}}{q+q^{-1}}\right) e_i &\text{ for $r$ even}.
\end{align*}
\end{lemma}

\begin{proof}First we prove for $r>0$ with induction on $r$.
For $r=1$, the statement is clearly true. For $r=2$:
$$
g_i^{2} = 1-e_i + \left( \frac{q^2-q^{-2}}{q+q^{-1}} \right)e_ig_i +\left(  \frac{q+q^{-1}}{q+q^{-1}}\right) e_i
= 1-e_i + \left( q-q^{-1} \right)e_ig_i + e_i
= 1 + \left( q-q^{-1} \right)e_ig_i.
$$
Suppose the statement is true for any $r \in \N$ up to $2k-1$. Then for $r=2k$, we have that:
\begin{align*}
g_i^{2k-1}g_i &= (1-e_i)\,g_i^2 + \left( \frac{q^{2k-1}+q^{-2k+1}}{q+q^{-1}} \right)e_ig_i^2 +\left(  \frac{q^{2k-2}-q^{-2k+2}}{q+q^{-1}}\right) e_i g_i
\\
&= (1-e_i) + (q-q^{-1})(1-e_i)\,e_ig_i^2 + \left( \frac{q^{2k-1}+q^{-2k+1}}{q+q^{-1}} \right)e_i \\
&\quad + \left( \frac{q^{2k-1}+q^{-2k+1}}{q+q^{-1}} \right)(q-q^{-1}) e_ig_i +\left(  \frac{q^{2k-2}-q^{-2k+2}}{q+q^{-1}}\right) e_i\,g_i
\\
&= (1-e_i) + (q-q^{-1})\underbrace{(e_i-e_i^2)}_{=0}\,g_i^2 + \left( \frac{q^{2k-1}+q^{-2k+1}}{q+q^{-1}} \right)e_i \\
&\quad + \left( \frac{(q^{2k-1}+q^{-2k+1})(q-q^{-1})+q^{2k-2}-q^{-2k+2}}{q+q^{-1}} \right) e_ig_i 
\\
&= (1-e_i) + \left( \frac{q^{2k}-q^{-2k}}{q+q^{-1}} \right) e_ig_i + \left( \frac{q^{2k-1}+q^{-2k+1}}{q+q^{-1}} \right)e_i.
\end{align*}
Also for $r=2k+1$, we have that:
\begin{align*}
g_i^{2k}g_i &= (1-e_i)g_i + \left( \frac{q^{2k}-q^{-2k}}{q+q^{-1}} \right)e_ig_i^2 +\left(  \frac{q^{2k-1}+q^{-2k+1}}{q+q^{-1}}\right) e_i g_i
\\
&= (1-e_i)g_i + \left( \frac{q^{2k}-q^{-2k}}{q+q^{-1}} \right)e_i + \left( \frac{q^{2k}-q^{-2k}}{q+q^{-1}} \right)(q-q^{-1})e_ig_i +\left(  \frac{q^{2k-1}+q^{-2k+1}}{q+q^{-1}}\right) e_i g_i
\\
&= (1-e_i)g_i + \left( \frac{(q^{2k}-q^{-2k})(q-q^{-1})+q^{2k-1}+q^{-2k+1}}{q+q^{-1}} \right)e_ig_i + \left( \frac{q^{2k}-q^{-2k}}{q+q^{-1}} \right)e_i
\\
&= (1-e_i)g_i + \left( \frac{q^{2k+1}+q^{-2k-1}}{q+q^{-1}} \right)e_ig_i + \left( \frac{q^{2k}-q^{-2k}}{q+q^{-1}} \right)e_i.
\end{align*}
\end{proof}
The algebra $\YH(q)$ has linear dimension $n!\ d^n$ and the elements $t_1^{k_1}t_2^{k_2}\dots t_n^{k_n}g_w$, where $0\leq k_1,\dots,k_n<d$ and the subscript $w$ runs over all elements of the permutation group $\mathfrak{S}_{n}$, form a linear basis for $\YH(q)$. This basis gives rise to an inductive linear basis, as follows: Let $w_{n+1}$ be a basis element in ${\rm Y}_{n+1}(q)$. Then either $w_{n+1} = w_n g_n g_{n-1} \ldots g_i$ or $w_{n+1} = t_{n+1}^k w_n$, where $w_n$ is an element of the basis of $\YH(q)$.

The Yokonuma--Hecke algebras were originally introduced  by T. Yokonuma \cite{yo}  in the representation theory of finite Chevalley groups and they are natural generalizations of the Iwahori--Hecke algebras ${\rm H}_n(q)$. Indeed, for $d=1$ all framings are zero, so the corresponding elements of ${\mathcal F}_n$ are identified with elements in $B_n$; also we have $e_i=1$, so the quadratic relation~\eqref{quadr} becomes the well--known quadratic relation of the algebra ${\rm H}_n(q)$:
\begin{equation}\label{hquadr}
g_i^2 = 1 + (q - q^{-1}) \, g_i \qquad (1\leq i\leq n-1).
\end{equation}
Thus, the algebra ${\rm Y}_{1,n}(q)$ coincides with the   algebra ${\rm H}_n(q)$. The Yokonuma--Hecke algebras can be also regarded as unipotent algebras in the sense of \cite{thi}. The representation theory of these algebras has been studied in \cite{thi} and  \cite{chpa}. In \cite{chpa} a completely combinatorial approach is taken to the subject.

\begin{note} \rm In this paper we will sometimes identify algebra monomials with their corresponding braid words. 
\end{note}

\begin{rem} \rm \label{oldpresent}
In the papers \cite{jula1, jula2, jula3, jula4, jula5}, the algebra $\YH(q)$ is defined with parameter $u$ and $\YH(u)$ is generated by the elements $\wt{g}_1, \ldots, \wt{g}_{n-1}$ and $t_1, \ldots, t_n$, satisfying the relations \eqref{framedgp} (where $\wt{g}_i$ corresponds to $\sigma_i$) and the quadratic relations:

\begin{center}
$\wt{g}_i^2 = 1 + (u - 1) \, e_i + (u - 1) \, e_i \wt{g}_i \qquad (1\leq i\leq n-1).$
\end{center}

\noindent The new presentation of $\YH(q)$ used in this paper was obtained in \cite{chpa} by taking $u:=q^2$ and $g_i := \wt{g}_i + (q^{-1}-1)\, e_i \wt{g}_i$ or, equivalently, $\wt{g}_i = g_i + (q-1) \, e_i g_i$.
\end{rem}

\section{Markov traces on the Yokonuma--Hecke algebras}\label{traces}
In this section we recall the definition of a unique Markov trace defined on the algebras $\YH(q)$, as well as a necessary condition on the trace parameters, needed for obtaining framed link invariants.

\subsection{The Juyumaya trace}
By the natural inclusions ${\mathcal F}_n \subset {\mathcal F}_{n+1}$, which induce the  inclusions $\YH(q) \subset {\rm Y}_{d,n+1}(q)$, and using the inductive bases of the algebras $\YH(q)$ we have:

\begin{thm}[\cite{ju}]\label{trace}
For $z$, $x_1, \ldots, x_{d-1}$ indeterminates over $\C$
there exists a unique linear map
$$
\Jtr : \bigcup_{n \geq 0} \YH(q) \longrightarrow \C [z, x_1, \ldots, x_{d-1}]
$$
 satisfying the rules:
$$
\begin{array}{crcll}
(1) & \Jtr(\a \b) & = & \Jtr(\b \a) & \qquad \a,\b\in\YH(q) \\
(2) & \Jtr(1) &  = & 1 & \qquad 1 \in\YH(q) \\
(3) & \Jtr(\a g_n) & = & z\, \Jtr(\a) & \qquad \a\in\YH(q) \quad (\text{Markov  property} ) \\
(4) & \Jtr(\a t_{n+1}^s) & = & x_s\, \Jtr(\a) & \qquad \a\in\YH(q) \quad (1\leq s\leq d-1).
\end{array}
$$
\end{thm}
Note that for $d=1$ the trace restricts to the first three rules and it coincides with the  Ocneanu trace on the Iwahori--Hecke algebras.

\begin{note} \rm In this paper we will sometimes write $\Jtr(\a)$ for a framed braid $\a \in {\mathcal F}_n$, by using the natural epimorphism of ${\mathcal F}_n$ onto $\YH$.
\end{note}

\subsection{The {\rm E}--system} 
Using the natural epimorphism of the framed braid group ${\mathcal F}_n$ onto $\YH(q)$, the trace $\Jtr$ and the Markov framed braid equivalence, comprising conjugation in the groups ${\mathcal F}_n$ and positive and negative stabilization and destabilization (see for example \cite{KS}), in \cite{jula2} the authors tried to obtain a topological invariant for framed links after the method of V.F.R.~Jones \cite{jo} (using for the algebra the presentation discussed in Remark~\ref{oldpresent}). This meant that $\Jtr$ would have to be normalized, so that the closed braids $\widehat{\a}$ and $\widehat{\a \s_n}$ $(\a\in {\mathcal F}_n)$ be assigned the same value of the invariant, and re-scaled, so that the closed braids $\widehat{\a \s_n^{-1}}$ and $\widehat{\a \s_n}$  $(\a\in {\mathcal F}_n)$ be also assigned the same value  of the invariant.  However, as it turned out, $\Jtr(\a g_n^{-1})$ does not factor through $\Jtr(\a)$, that is: 
\begin{equation}\label{nofactor}
\Jtr(\a g_n^{-1}) 
\stackrel{(\ref{invrs})}{=}
\Jtr(\a g_n) - (q - q^{-1})\, \Jtr(\a e_n) \neq \Jtr(g_n^{-1}) \Jtr(\a).
\end{equation}
The reason is that, although $\Jtr(\a g_n) = z\, \Jtr(\a)$, yet 
 $\Jtr(\a e_{n})$ does not factor through $\Jtr(\a)$, that is:
 \begin{equation}\label{aen}
\Jtr(\a e_{n}) \neq \Jtr(e_{n}) \Jtr(\a),
\end{equation}
which is due to the fact that:
\begin{equation}\label{atn}
\Jtr(\a t_n^k) \neq \Jtr(t_n^k) \Jtr(\a) \quad k=1,\ldots,d-1.
\end{equation}
Forcing 
\begin{center}
$\Jtr(\a e_{n}) = \Jtr(e_{n}) \Jtr(\a)$
\end{center}
yields that the trace parameters $x_1, \ldots, x_{d-1}$ have to satisfy the ${\rm E}$-\emph{system}, the non-linear system of equations in $\C$:
\begin{equation}\label{Esystem}
E^{(m)} = {x}_m E \qquad (1\leq m \leq d-1)
\end{equation}
where
$$
E := E^{(0)} =\frac{1}{d}\sum_{s=0}^{d-1}{x}_{s}{x}_{d-s} = \Jtr(e_{i}) \qquad \mbox{and} \qquad  E^{(m)} :=\frac{1}{d} \sum_{s=0}^{d-1}{x}_{m+s}{x}_{d-s} \, ,
$$
where the sub-indices on the ${x}_j$'s are regarded modulo $d$ and ${x}_0:=1$ (see \cite{jula2}). As it was shown  by P.~G\'{e}rardin (in the Appendix of \cite{jula2}), the solutions of the {\rm E}--system are parametrized by the non-empty subsets of $\Z/d\Z$. For example, for every singleton subset $\{ m \}$ of $\Z/d\Z$, we have a solution of the ${\rm E}$--system given by:
\begin{equation}\label{Solution singleton}
{\rm x}_1 ={\rm exp}({{2 \pi m \sqrt{-1}}/{d}}) \qquad \text{ and } \qquad {\rm x}_k = {\rm x}_1^k  \,\,\,\,\, \text{ for }  k=2,\ldots, d-1.
\end{equation}

\subsection{The specialized trace}\label{strace} 

Let  $X_D := ({\rm x}_1, \ldots , {\rm x}_{d-1})$ be a solution of the {\rm E}--system parametrized by the non-empty subset $D$ of $\Z /d \Z$. We shall call {\it specialized trace}  the trace $\Jtr$ with the parameters $x_1,\ldots,x_{d-1}$ specialized to ${\rm x}_1, \ldots , {\rm x}_{d-1} \in \C$, and it shall be denoted $\Jtrs$ (cf. \cite{chla}). More precisely,
\begin{center}
$
\Jtrs : \bigcup_{n}\YH(q) \longrightarrow \C [z] 
$
\end{center}
is a Markov trace on the Yokonuma--Hecke algebra, satisfying the following rules:
$$
\begin{array}{ccccl}
(1) & \Jtrs(\a \b) & = & \Jtrs(\b \a) & \qquad \a,\b\in\YH(q) \\
(2) & \Jtrs(1) &  = & 1 & \qquad 1 \in\YH(q) \\
(3) & \Jtrs(\a g_n) & = & z\, \Jtrs(\a) & \qquad \a\in\YH(q) \quad (\text{Markov  property} ) \\
(4^\prime) & \Jtrs (\a t_{n+1}^s) & = & {\rm x}_s\,\Jtrs (\a) & \qquad \a\in\YH(q) \quad (1\leq s\leq d-1).
\end{array}
$$
\noindent
The rules (1)--(3) are the same as in Theorem~\ref{trace}, while rule (4) is replaced by the rule ($4^\prime$). As it turns out \cite{jula3}:  
$$
E_D := \Jtrs(e_{i}) = \frac{ 1}{\vert D\vert},
$$
where $\vert D\vert$ is the cardinal of the subset $D$. Note that  ${\rm tr}_{1, \{ 0\}}$  coincides with ${\rm tr}_1$ which in turn coincides with the Ocneanu trace $\tau$.

\section{Properties of the Markov traces}\label{properties}
We shall now give some properties of the traces $\Jtr$ and $\Jtrs$, analogous to known properties of the Ocneanu trace $\tau$, by considering their behaviour under the  operations below. Clearly, a property satisfied by $\Jtr$ is also satisfied by $\Jtrs$ (and by $\tau$), but the converse may not hold.

\subsection{Inversion of braid words} Inversion means that a braid word is written from right to left. 
For $\alpha=t_1^{k_1}\ldots t_n^{k_n} \s_{i_1}^{l_1} \ldots \s_{i_r}^{l_r} \in \mathcal{F}_n$, where $k_1,\ldots,k_n,l_1,\ldots,l_r \in \Z$, we denote by
$\overleftarrow{\a}$ the inverted word, that is, 
  $\overleftarrow{\a} = \s_{i_r}^{l_r} \ldots \s_{i_1}^{l_1}t_n^{k_n}\ldots t_1^{k_1}$.
On the level of closed braids this operation corresponds to the change of orientation on all components of the resulting link. On the algebra level inversion is defined on any word by linear extension.  We will see that the trace $\Jtr$ is invariant under this operation. Indeed:

\begin{prop}\label{invProp}
Let $a \in \YH(q)$. Then $\Jtr(a) = \Jtr(\overleftarrow{a})$. In particular, $\Jtr(\a) = \Jtr(\overleftarrow{\a})$ for all $\a \in  \mathcal{F}_n$.
\end{prop}

\begin{proof}
First, note that the linear map sending $a$ to $\overleftarrow{a}$ defines a $\C$-algebra antiautomorphism on $\YH(q)$, since it respects all defining relations of the Yokonuma--Hecke algebra (relations
\eqref{framedgp}, \eqref{modular} and \eqref{quadr}). Thus, we have 
\begin{equation}\label{arrow multi}
\overleftarrow{ab}=\overleftarrow{b}\overleftarrow{a} \quad \text{for all $a,b \in \YH(q)$}.
\end{equation}

For the proof of the proposition we proceed by induction on $n$. For $n=1$ the statement holds trivially. Assume now that the statement is true for all elements in $\YH(q)$. Since $\Jtr$ is a linear map, it is enough to prove the statement for all elements of the inductive basis $\mathcal{B}_{n+1}^{\rm ind}$ of ${\rm Y}_{d,n+1}(q)$.
Let $w_n \in \mathcal{B}_{n}^{\rm ind}$, $k \in \Z/d\Z$ and $i \in \{1,\ldots,n\}$. We have 
 $$ \Jtr(w_n g_n g_{n-1} \ldots g_i)
= z\, \Jtr(w_n g_{n-1} \ldots g_i)
= z\, \Jtr(\overleftarrow{w_n g_{n-1} \ldots g_i})$$
$$ = z\, \Jtr(g_i \ldots g_{n-1} \overleftarrow{w_n})
= \Jtr(g_i \ldots g_{n-1} g_n \overleftarrow{w_n})
= \Jtr(\overleftarrow{w_n g_n g_{n-1} \ldots g_i}) ,$$
where we have used rules (1) and (3) of the definition of $\Jtr$ for the first and fourth equalities, the induction hypothesis for the second equality and \eqref{arrow multi} for the third and fifth equalities, and
$$\Jtr(w_nt_{n+1}^k )
= x_k\, \Jtr(w_n)
= x_k\, \Jtr(\overleftarrow{w_n})
= \Jtr(t_{n+1}^k \overleftarrow{w_n} )
= \Jtr(\overleftarrow{  w_nt_{n+1}^k}),$$
where we have used rules (1) and (4) of the definition of $\Jtr$ for the first and third equalities, the induction hypothesis for the second equality and \eqref{arrow multi} for the last equality.
\end{proof}

\subsection{Split links} Let $L = L_1 \sqcup \ldots \sqcup L_m$ be a split link, where $L_1, \ldots, L_m$ are links. Then there exists a braid word $\a = \a_1 \ldots \a_m \in B_n$, where $\a_i \in B_{i_j} \setminus B_{i_{j-1}+1}$ for some $1 \leq i_1 < \ldots < i_k \leq n$ with $i_{j+1}-i_j>1$ such that $\widehat{\a} = L$ and $\widehat{\a_i} = \sqcup_{k=1}^{i_{j-1}} U \sqcup L_i$ for each $i=1,\ldots,m$. Then the traces $\Jtr$ and $\Jtrs$ are multiplicative, as follows:
\begin{prop}\label{split_trace}
Let $L = \widehat{\a} = \widehat{\a_1 \ldots \a_m}$ be a split link as above. Then:
$$
\Jtr(\a) = \Jtr(\a_1) \cdots \Jtr(\a_m),
$$
and consequently:
$$
\Jtrs(\a) = \Jtrs(\a_1) \cdots \Jtrs(\a_m).
$$
\end{prop}
\begin{proof}
The proof is immediate, since when calculating the trace $\Jtr(\a_1 \ldots \a_m)$ we exhaust first the word $\a_m$ and so $\Jtrs(\a) = \Jtr(\a_1 \ldots \a_{m-1}) \Jtr(\a_m)$. Iterating this procedure, we obtain the result.
\end{proof}

\subsection{Connected sums}  Let $\a \in {\mathcal F}_n$ and $\b \in {\mathcal F}_m$ for some $n,m \in \N$. The \emph{connected sum} of $\a$ and $\b$ is the word $\a \# \b: = \a^{[0]} \b^{[n-1]} $ in the framed braid group ${\mathcal F}_{n+m-1}$, where $\a^{[0]}$ is the natural embedding of $\a$ in ${\mathcal F}_{n+m-1}$, while  $\b^{[n-1]}$ is the  embedding of $\b$ in ${\mathcal F}_{n+m-1}$ induced by the following shifting of the indices: $\s_i \mapsto \s_{n+i-1}$ for $i \in \{1,\ldots,m-1\}$ and $t_j \mapsto t_{n+j-1}$ for $j \in \{1,\ldots,m\}$. Upon closing the braids, this operation corresponds to taking the connected sum of the resulting framed links.

It is known that the Ocneanu trace is multiplicative under the connected sum operation, that is, $\tau(\a \# \b ) = \tau(\a) \, \tau(\b)$ if $\a \in B_n$ and $\beta \in B_m$.
On the other hand, the trace $\Jtr$ is not multiplicative under the connected sum operation, due to \eqref{atn} and \eqref{aen} (we have $\a \# t_1^k = \alpha t_n^k$ and $\a \# e_1 = \a e_n$). Yet, as we will see, the specialized trace $\Jtrs$ is  multiplicative on connected sums, due to the E--condition, but this is only true on the level of classical braids. For framed braids this is true only when $E_D=1$, that is, when $D$ is singleton and hence the corresponding solution $X_D$ of the E--system is described by (\ref{Solution singleton}). Indeed, for classical braids we have:

\begin{prop} \label{connsum}
Let $\a \in B_n$ and $\beta \in B_m$. Then $\Jtrs(\a \# \b ) = \Jtrs(\a) \, \Jtrs(\b)$.
\end{prop}

\begin{proof}
By the definition of $\a \# \b$, $\a^{[0]}$ is a word containing the generators $\s_1, \ldots, \s_{n-1}$ and $\b^{[n-1]}$ is a word containing the generators $\s_n, \ldots, \s_{n+m-2}$. Therefore,  in computing $\Jtrs (\a \# \b )$ we will first exhaust the word $\b^{[n-1]}$.  Further, since the polynomial coefficients of the quadratic relation \eqref{quadr} do not depend on the index of the generator $g_i$ and since $\Jtrs$ factors through positive and negative stabilization, we have that $\Jtrs(\b) = \Jtrs(\b^{[n-1]})$. Hence,
\begin{center}
$\Jtrs (\a \# \b ) = \Jtrs(\a^{[0]} \b^{[n-1]}) = \Jtrs(\a) \, \Jtrs(\b^{[n-1]}) = \Jtrs(\a)\, \Jtrs(\b).$
\end{center}
\end{proof}

\begin{rem} \rm
Proposition~\ref{connsum} holds independently of the component we choose to connect the two braids. 
Indeed, connecting $\b$ to a different component of $\a$ means that we connect $\b$ to a conjugate word of $\a$. Let $\wt{a}$ be such a word. Then:
$$\Jtrs (\a \# \b ) = \Jtrs(\a^{[0]}) \Jtrs(\b^{[n-1]}) = \Jtrs(\wt{a}^{[0]}) \, \Jtrs(\b^{[n-1]}) = \Jtrs(\wt{a} \# \b).$$
\end{rem}

For framed braids we have the following:

\begin{prop}
Let $\a \in {\mathcal F}_n$ and $\b \in {\mathcal F}_m$. Then $\Jtrs(\alpha \# \beta) = \Jtrs(\alpha)\Jtrs(\beta)$ if and only if  ${\rm x}_1^d=1$ and ${\rm x}_k={\rm x}_1^k$ for all $k=1,\ldots,d-1$.
 \end{prop}
 
\begin{proof}
Let $\alpha = t_1^k$ and $\beta=t_1^l$ for $k,l \in \Z/d\Z$. If $\Jtrs(\alpha \# \beta) = \Jtrs(\alpha)\Jtrs(\beta)$ for all $\alpha,\beta$, then 
$\Jtrs(t_1^{k+l}) = \Jtrs(t_1^k)\Jtrs(t_1^l)$, and so ${\rm x}_{k+l}={\rm x}_k {\rm x}_l$. We deduce that ${\rm x}_1^d={\rm x}_d={\rm x}_0=1$ and ${\rm x}_k={\rm x}_1^k$  for all $k=1,\ldots,d-1$.

Suppose now that  ${\rm x}_1^d=1$ and ${\rm x}_k={\rm x}_1^k$ for all $k=1,\ldots,d-1$. In this case, as we have shown in \cite[\S 3.4]{chla}, the map
$$
\begin{array}{rccc}
\varphi :   \ {\rm Y}_{d,n}(q) & \longrightarrow & {\rm H}_n(q) & \\
 g_i  &  \mapsto & G_i &  \\
 t_j^k   & \mapsto & {\rm x}_k &(1 \leq k \leq d-1)
\end{array}
$$ 
is a $\C$-algebra epimorphism, and we have
\begin{equation*}
 \tau \circ \varphi = \Jtrs .
\end{equation*}

Let $\alpha=t_1^{k_1}\ldots t_n^{k_n}\,\wt{\a}$, with $\wt{\a} \in B_n$, and $\beta = t_1^{l_1}\ldots t_m^{l_m}\,\wt{\b}$, with $\wt{\b} \in B_m$.
We have $\a \# \b = \a^{[0]} \b^{[n-1]}$, where $\a^{[0]} = t_1^{k_1}\ldots t_n^{k_n}\,\wt{\a}^{[0]}$, $\b^{[n-1]} = t_n^{l_1}\ldots t_{n+m-1}^{l_{m}}\,\wt{\b}^{[n-1]}$
and $\wt{\a}^{[0]}\wt{\b}^{[n-1]} = \wt{\a}\# \wt{\b}$ in the classical braid group $B_{n+m-1}$. 
Then $$\begin{array}{rcl}
\Jtrs ( \a^{[0]} \b^{[n-1]}) & = &  (\tau \circ \varphi) (\a^{[0]} \b^{[n-1]}) \\
&= & \tau({\rm x}_{k_1}\ldots {\rm x}_{k_n}\wt{\a}^{[0]}{\rm x}_{l_1}\ldots {\rm x}_{l_{m}}\wt{\b}^{[n-1]})\\
& =& {\rm x}_{k_1}\ldots {\rm x}_{k_n}{\rm x}_{l_1}\ldots {\rm x}_{l_{m}} \tau(\wt{\a}^{[0]}\wt{\b}^{[n-1]})\\
& = &{\rm x}_{k_1}\ldots {\rm x}_{k_n}{\rm x}_{l_1}\ldots {\rm x}_{l_{m}} \tau(\wt{\a})\tau(\wt{\b})  \\
& =  & \tau({\rm x}_{k_1}\ldots {\rm x}_{k_n}\wt{\a})\tau({\rm x}_{l_1}\ldots {\rm x}_{l_m} \wt{\b}) \\
& = &( \tau \circ \varphi)(\alpha) (\tau \circ \varphi)(\beta),
\end{array} $$
because $\tau$ has the multiplicative property on connected sums.
\end{proof}

\subsection{Mirror images} Let us consider the group automorphism of $B_n$ given by $\s_i \mapsto \s_i^{-1}$. For $\alpha \in B_n$, we denote by $\alpha^*$ the image of $\alpha$ via this automorphism. We call $\alpha^*$ the \emph{mirror image} of $\alpha$.
On the level of closed braids this operation corresponds to switching all crossings. 

The following result is known as the ``mirroring property'' in the case of the Ocneanu trace. Again, due to \eqref{nofactor}, the trace  $\Jtr$  does not satisfy this property, but the specialized trace $\Jtrs$ does.

Before formulating our results, we observe that $\Jtrs(g_i^{-1}) = z - (q-q^{-1})E_D$ and we introduce a new variable $\lambda_D$ in place of $z$. Namely, by re-scaling $\s_i$ to $\sqrt{\lambda_D} g_i$, so that $\Jtrs(g_n^{-1}) =\lambda_D z$, we find
\begin{equation}\label{Dlambda}
\lambda_D := \frac{z - (q-q^{-1})E_D}{z}.
\end{equation}
If we solve \eqref{Dlambda} with respect to the variable $z$, we obtain
$$
z= \frac{(q-q^{-1})\,E_D}{1-\lambda_D}.
$$
Hence, the trace $\Jtrs$ can be considered as a polynomial in the variables $(q,z)$ or the variables $(q,\lambda_D)$ by the above change of variables.

\begin{prop}\label{mirror_trace}
Let $\a \in B_n$.
Then
$$
\Jtrs(q,z) (\a^*)= \Jtrs\left(q^{-1},z - (q-q^{-1})E_D\right)(\a).
$$
or equivalently,
$$
\Jtrs (q,\lambda_D) (\a^*)=\Jtrs\left(q^{-1}, \lambda_D^{-1} \right)(\a) .
$$
\end{prop}

\begin{proof}
Set $\check{g}_i:=g_i^{-1}$, for all $i=1,\ldots,n-1$, and $\check{q} := q^{-1}$. It is easy to chek that the algebra $\YH(q)$ is generated by the elements $\check{g}_1, \ldots, \check{g}_{n-1}, t_1,\ldots,t_n$, satisfying relations (\ref{framedgp}) and (\ref{modular}), together with the quadratic relations:
\begin{center}
$\check{g}_i^2 = g_i^{-2} = 1 - (q-q^{-1})e_i g_i^{-1} = 1 + (\check{q}-\check{q}^{-1})e_i \check{g}_i$.
\end{center}
Hence, the quadratic relation does not change under the new presentation. 

Using the new presentation, we can now define a specialized trace $\check{\rm tr}_{d,D}(\check{q}, \check{z})$ on $\YH(q)$ with parameter $\check{z}$.  We have $\check{\rm tr}_{d,D} (\check{q}, \check{z})= \Jtrs(\check{q}, \check{z})$ if and only if 
$ \check{z} =  \Jtrs(g_i^{-1}) = z-(q-q^{-1})E_D$.

It is now easily seen that calculating the trace of a word $\a$ under the new presentation (in variables $(\check{q}, \check{z})$) corresponds to calculating the trace of the word $\a^*$ with the original presentation (in variables $(q,z)$), since the two traces will be the same polynomials (in different variables). We deduce that
$$\Jtrs(q,z)(\a^*) =  \Jtrs(\check{q}, \check{z})(\a)=\Jtrs\left(q^{-1},z - (q-q^{-1})E_D\right)(\a) .$$

Applying the transformations $q\mapsto q^{-1}$ and $z \mapsto (z - (q-q^{-1})E_D)$ on $\lambda_D$, we obtain:
$$
\check{\lambda}_D = \frac{\check{z} - (\check{q}-\check{q}^{-1})E_D}{\check{z}} = 
\frac{z - (q-q^{-1})E_D - (q^{-1}-q)E_D}{z - (q-q^{-1})E_D} = \frac{z}{z - (q-q^{-1})E_D} = \lambda_D^{-1},
$$
which proves the result for the variables $(q,\lambda_D)$.
\end{proof}

\section{Link invariants from the Yokonuma--Hecke algebras}\label{yhinvts}
Given a solution $X_D := ({\rm x}_1, \ldots , {\rm x}_{d-1})$ of the {\rm E}--system, then from $\Jtrs$ invariants for various types of knots and links, such as framed, classical and singular, have been constructed in \cite{jula2,jula3,jula4}. We shall now adapt the definition of these invariants in view of the new presentation of $\YH(q)$ used in this paper. Then we shall define an invariant for transverse knots.

\subsection{Framed links} 
Let ${\mathcal L}_f$ denote the set of oriented  framed links. We set:
\begin{equation}\label{CapitalLambda}
\Lambda_D:=\frac{1}{z \sqrt{\lambda_D}}.
\end{equation}
From the above and re-scaling $\s_i$ to $\sqrt{\lambda_D} g_i$, so that $\Jtrs(g_n^{-1}) =\lambda_D z$, we have the following Theorem, which is analogous to \cite[Theorem 8]{jula2}:

\begin{thm}\label{framedinv} 
 Given a solution $X_D$ of the {\rm E}--system, for any framed braid $\a \in {\mathcal  F}_{n}$ we define for the framed link $\widehat{\a} \in {\mathcal L}_f$:
$$
\Phi_{d,D} (\widehat{\a}) = \Lambda_D^{n-1} (\sqrt{\lambda_D})^{\epsilon(\a)} \left(\Jtrs \circ \gamma \right)(\a)
$$
where $\gamma : \C {\mathcal  F}_{n}  \longrightarrow  \YH(q)$ is the natural algebra homomorphism defined via: $\s_i \mapsto g_i$ and $t_j^s  \mapsto t_j^{s( {\rm mod}\, d)}$,  and $\epsilon(\a)$ is the algebraic sum of the exponents of the $\s_i$'s in $\a$.
 Then the map  $\Phi_{d,D} (q,z)$  is a 2-variable isotopy invariant of oriented framed links.
\end{thm}

\begin{prop}The invariant $\Phi_{d,D}$ satisfies the following skein relation:
\begin{equation*}
\frac{1}{\sqrt{\lambda_D}} \Phi_{d,D}(L_+) - \sqrt{\lambda_D} \Phi_{d,D}(L_-) = \frac{q-q^{-1}}{d} \, \sum_{s=0}^{d-1} \Phi_{d,D}(L_s),
\end{equation*}
where the links $L_+$, $L_-$ and $L_s$ are closures of the framed braids illustrated in Figure~\ref{framed_skein}.
\end{prop}

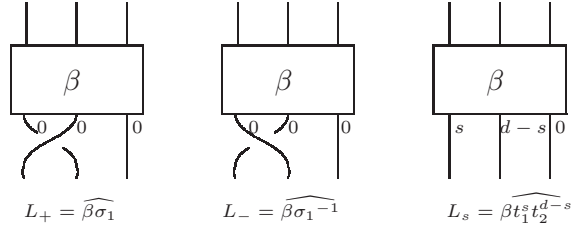
\begin{figure}[H]
\begin{center}
\begin{picture}(216,80)
\put(20,53){$\beta$}

\qbezier(6,70)(6,78)(6,86) 
\qbezier(25,70)(25,78)(25,86)
\qbezier(44,70)(44,78)(44,86)

\qbezier(5,20)(4,22)(5,24) 
\qbezier(25,20)(26,22)(25,24) 

\qbezier(0,44)(0,59)(0,70)
\qbezier(50,44)(50,59)(50,70)

\qbezier(0,70)(25,70)(50,70)
\qbezier(0,44)(25,44)(50,44)

\qbezier(5,24)(5,29)(15,34)
\qbezier(15,34)(25,39)(25,44)
\qbezier(25,24)(25,28)(20,31)
\qbezier(10,37)(5,40)(5,44)

\qbezier(44,20)(44,32)(44,44)

\put(100,53){$\beta$}

\qbezier(86,70)(86,78)(86,86) 
\qbezier(105,70)(105,78)(105,86)
\qbezier(124,70)(124,78)(124,86)

\qbezier(85,20)(84,22)(85,24) 
\qbezier(105,20)(106,22)(105,24) 

\qbezier(80,44)(80,59)(80,70)
\qbezier(130,44)(130,59)(130,70)

\qbezier(80,70)(105,70)(130,70)
\qbezier(80,44)(105,44)(130,44)

\qbezier(85,24)(85,28)(90,31)
\qbezier(100,37)(105,40)(105,44)
\qbezier(105,24)(105,29)(95,34)
\qbezier(95,34)(85,39)(85,44)

\qbezier(124,20)(124,32)(124,44)

\put(180,53){$\beta$}

\qbezier(166,70)(166,78)(166,86) 
\qbezier(185,70)(185,78)(185,86)
\qbezier(204,70)(204,78)(204,86)

\qbezier(160,44)(160,59)(160,70)
\qbezier(210,44)(210,59)(210,70)

\qbezier(160,70)(185,70)(210,70)
\qbezier(160,44)(215,44)(210,44)

\qbezier(166,20)(166,32)(166,44) 
\qbezier(185,20)(185,32)(185,44)
\qbezier(204,20)(204,32)(204,44)

\put(10,37){\tiny{$0$}}
\put(25,37){\tiny{$0$}}
\put(46,37){\tiny{$0$}}

\put(90,37){\tiny{$0$}}
\put(105,37){\tiny{$0$}}
\put(125,37){\tiny{$0$}}

\put(168,37){\tiny{$s$}}
\put(185,37){\tiny{$d-s$}}
\put(206,37){\tiny{$0$}}

\put(5,5){\tiny{$L_{+}=\widehat{\beta \sigma_1}$}}
\put(80,5){\tiny{$L_{-}=\widehat{\beta {\sigma_1}^{-1}}$}}
\put(165,5){\tiny{$L_s=\widehat{\beta t_1^st_2^{d-s}}$}}

\end{picture}
\caption{The framed links in the skein relation in open braid form.}
\label{framed_skein}
\end{center}
\end{figure}

\begin{proof}
The proof is after Jones' method \cite{jo}. Namely, by the Alexander theorem for framed links we may assume that $L_+$ is in braided form and that $L_+ = \widehat{\b \s_i}$ for some $\b \in \mathcal{F}_n$ and for some $i$. Then:
$$
L_- = \widehat{\b \s_i^{-1}}, L_s = \widehat{\b t_i^s t_{i+1}^{d-s}} \quad (s=0,\ldots,d-1).
$$
We now apply relation \eqref{invrs}. Since $\epsilon(\b \s_i^{-1})=\epsilon(\b) -1 $, $\epsilon(\b \s_i)=\epsilon(\b) + 1$, $\epsilon(\b t_i^s t_{i+1}^{d-s}) = \epsilon(\b)$, we obtain:
\begin{align*}
\Phi_{d,D}(L_-) = &\Lambda_D^{n-1} \sqrt{\lambda_D}^{\epsilon(\b \s_i^{-1})} \Jtrs(\b g_i^{-1})
\\
\stackrel{\eqref{invrs}}{=} &\Lambda_D^{n-1} \sqrt{\lambda_D}^{\epsilon(\b)-1} \left[ \Jtrs(\b g_i) - (q-q^{-1}) \Jtrs(\b e_i) \right]
\\
= &\Lambda_D^{n-1} \sqrt{\lambda_D}^{\epsilon(\b)} \left[ \frac{1}{\sqrt{\lambda_D}}\Jtrs(\b g_i) - \frac{q-q^{-1}}{\sqrt{\lambda_D}} \Jtrs(\b e_i) \right]
\\
\stackrel{\eqref{ei}}{=} &\frac{1}{\lambda_D} \Lambda_D^{n-1} \sqrt{\lambda_D}^{\epsilon(\b)+1} \Jtrs(\b g_i) -  \frac{q-q^{-1}}{d \sqrt{\lambda_D}} \Lambda_D^{n-1} \sqrt{\lambda_D}^{\epsilon(\b)} \sum_{s=0}^{d-1} \Jtrs(\b t_i^s t_{i+1}^{d-s}) 
\\
= &\frac{1}{\lambda_D} \Phi_{d,D}(L_+) - \frac{q-q^{-1}}{d \sqrt{\lambda_D}} \sum_{s=0}^{d-1} \Phi_{d,D}(L_s).
\end{align*}
\end{proof}

\begin{rem}	
{\rm Note that, for every $d\in \mathbb{N}$, we have $2^d-1$ distinct solutions of the {\rm E}--system, so the above construction yields $2^d -1$ isotopy invariants for framed links.}\end{rem}

\begin{rem} \rm
The behaviour of the trace $\Jtrs$ under inversion, split links and connected sums (when ${\rm x}_1^d=1$ and ${\rm x}_k={\rm x}_1^k$ for all $k=1,\ldots,d-1$) carries through to the invariants $\Phi_{d,D}$. Especially about split links, we have the following multiplicative property of the invariants $\Phi_{d,D}$, which was already known for the Homflypt polynomial.
\end{rem}

\begin{prop}
Let $L = L_1 \sqcup \ldots \sqcup L_m$ be a split link, where $L_1, \ldots, L_m$ are links. Then:
\begin{center}
$\Phi_{d,D}(L) = \Lambda_D^{m-1} \Phi_{d,D}(L_1) \ldots \Phi_{d,D}(L_m).$
\end{center}
\end{prop}

\begin{proof}
Since $L = L_1 \sqcup \ldots \sqcup L_m$, it can be written in braid form as $\a_1 \ldots \a_m$, where $\widehat{\a_i} = L_i$, for $i=1,\ldots,m$ as in Proposition~\ref{split_trace}. Suppose that each $\a_i$ has $\ell_i$ strands, so that $\ell_1 + \ldots + \ell_m = n$. By Proposition~\ref{split_trace} we have that:
\begin{align*}
\Phi_{d,D}(L) &= \Lambda_D^{n-1} \sqrt{\lambda_D}^{\epsilon(\a_1 \ldots \a_m)} \Jtrs(\a_1 \ldots \a_m)
= \Lambda_D^{\ell_1 + \ldots + \ell_m -1} \sqrt{\lambda_D}^{\epsilon(\a_1 \ldots \a_m)} \Jtrs(\a_1) \ldots \Jtrs(\a_m)\\
&= \Lambda_D^{m-1} \Phi_{d,D}(L_1) \ldots \Phi_{d,D}(L_m),
\end{align*}
which proves the result.
\end{proof}

\subsection{Classical links}
Let  ${\mathcal L}$ denote the set of oriented classical links.  The classical braid group $B_n$ injects into the framed braid group ${\mathcal F}_n$, whereby elements in $B_n$ are viewed as framed braids with all framings zero. So, by the classical Markov braid equivalence, comprising conjugation in the groups $B_n$ and positive and negative stabilizations and destabilizations, and by the construction and notations above, we obtain isotopy invariants for oriented classical knots and links, where the $t_j$'s are treated as formal generators. These invariants of classical links, which are analogous to those defined in  \cite{jula3} where the old presentation for the Yokonuma--Hecke algebra is used, are denoted as $\Theta_{d,D}$ and the restriction of $\gamma : \C {\mathcal  F}_{n}  \longrightarrow  \YH(q)$ on $\C B_n$ is denoted as $\d$. Namely,
$$
\Theta_{d,D} (\widehat{\a}) := \Lambda_D^{n-1} (\sqrt{\lambda_D})^{\epsilon(\a)}\, (\Jtrs \circ \d) (\a).
$$

\smallbreak
The invariants $\Theta_{d,D}(q,z)$ need to be compared with known invariants of classical links, especially with the Homflypt polynomial. The Homflypt polynomial $P (q,z)$ is a 2-variable isotopy invariant of oriented classical links that was constructed from the Iwahori--Hecke algebras ${\rm H}_n(q)$ and the Ocneanu trace $\tau$ after re--scaling and normalizing $\tau$ \cite{jo}. In this paper we define $P$ via the invariants $\Theta_{d,D}$, since for $d=1$ the algebras  ${\rm H}_n(q)$ and ${\rm Y}_{1,n}(q)$ coincide and the traces $\tau$, ${\rm tr}_1$ and ${\rm tr}_{1, \{ 0\}}$ also coincide. Namely, we define:
$$
P(\widehat{\a}) = \Theta_{1,\{0\}}(\widehat{\a}) = \left( \frac{1}{z \sqrt{\lambda_{\rm H}}} \right)^{n-1} (\sqrt{\lambda_{\rm H}})^{\epsilon(\a)} \left({\rm tr}_{1, \{ 0\}} \circ \delta \right)(\a)
$$
where $\lambda_{\rm H} := \frac{z-(q-q^{-1})}{z} = \lambda_{\{0\}}$. Further, recall that the Homflypt polynomial satisfies the following skein relation \cite{jo}:
$$
\frac{1}{\sqrt{\lambda_{\rm H}}} \, P(L_+) - \sqrt{\lambda_{\rm H}} \, P(L_-) = (q-q^{-1}) \, P(L_0)
$$
where $L_+, L_-, L_0$ is a Conway triple.

\smallbreak
 Contrary to the case of framed links, the skein relation of the invariants $\Phi_{d,D}(q,z)$ has no topological interpretation in the case of classical links, since it introduces framings.  This makes it very difficult to compare the invariants  $\Theta_{d,D} (q,z)$ with the Homflypt polynomial using diagrammatic methods. 
On the algebraic level, there are no algebra homomorphisms connecting the algebras and the traces \cite{chla}.  Consequently, in \cite{chla} it is shown that for generic values of the parameters $q,z$ the invariants $\Theta_{d,D}(q,z)$ do not coincide with the Homflypt polynomial. They only coincide in  the trivial cases where $q=\pm 1$ or $\Jtr(e_i)=1$. The last case implies that the solution of the {\rm E}--system comprises the $d$-th roots of unity.

\smallbreak
Yet, our computational data below indicate that these invariants do not distinguish more or less knot pairs than the Homflypt polynomial. So, the invariants $\Theta_{d,D}(q,z)$ need to be further investigated, as they may be topologically equivalent to the Homflypt polynomial $P(q,z)$. 

\begin{rem} \rm
The behaviour of the traces $\Jtrs$ under inversion, split links, connected sums and mirror imaging carries through to the invariants $\Theta_{d,D}$. Especially about mirror images, we have the following mirroring property of the invariants $\Theta_{d,D}$, which is already known to be valid for the Homflypt polynomial.
\end{rem}
\begin{prop}\label{mirror}
Given a solution $X_D$ of the {\rm E}--system, for any braid $\a \in B_{n}$, we have 
$$
\Theta_{d,D}(q,z) (\widehat{\a^*}) = \Theta_{d,D}  \left({q}^{-1},z - (q-q^{-1})E_D\right)(\widehat{\a}),
$$
or, equivalently, in a way analogous to  the Homflypt polynomial,
$$
\Theta_{d,D}(q,\lambda_D)  (\widehat{\a^*})= \Theta_{d,D}  \left(q^{-1}, \lambda_D^{-1}  \right) (\widehat{\a}).
$$
\end{prop}

\begin{proof}
It is easy to check that the quantity $\Lambda_D$ is invariant under the transformation 
\begin{center}
$q \mapsto {q}^{-1}$ \,\,\,\,\,and\,\,\,\,\, $\lambda_D \mapsto{\lambda_D}^{-1}$. 
\end{center}
Now, for any  $\a \in B_{n}$, we have $\epsilon(\a^*) = -\epsilon(\a)$, and so,
\begin{align*}
\Theta_{d,D}(q,\lambda_D)  (\widehat{\a^*})& \,\,\,= \Lambda_D^{n-1} (\sqrt{\lambda_D})^{-\epsilon(\a)} \,\Jtrs(q,\lambda_D) (\a^*)
\\
&= \Lambda_D^{n-1} \left (\sqrt{\frac{1}{\lambda_D}} \right)^{\epsilon(\a)} \Jtrs \left(q^{-1}, \lambda_D^{-1}\right) (\a) = \Theta_{d,D} \left(q^{-1}, \lambda_D^{-1}\right)(\widehat{\a}),
\end{align*}
where for the last equality we used Proposition~\ref{mirror_trace}.
\end{proof}

\subsection{Singular links}

Let ${\mathcal L}_{\mathcal S}$ denote the set of oriented singular links.  Oriented singular links are represented by singular braids, which form the singular braid monoids  ${\mathcal S}B_n$ \cite{ba,bi,sm}.  The singular braid monoid ${\mathcal S}B_n$ is generated by the classical braiding generators $\s_i$ with their inverses, together with the elementary singular braids $\tau_i$, which are not invertible.  In \cite{jula4} a monoid homomorphism was constructed. Here we give another one adapted to the new quadratic relation, namely:
\begin{equation}\label{eta}
\begin{array}{cccl}
\eta : & {\mathcal S} B_n & \longrightarrow & \YH(q) \\
 & \s_i  & \mapsto & g_i \\
 & \tau_i    & \mapsto & e_i
\end{array}
\end{equation}

Using now  the singular braid equivalence \cite{ge}, the map $\eta$ and the specialized  trace $\Jtrs$ we obtain isotopy invariants for oriented singular links, analogous to the ones constructed in \cite[Theorem 3.6]{jula4}, as follows:

\begin{thm}\label{Jula4:Th5}
For any singular braid $\a \in {\mathcal S}B_n$, we define
$$
\Psi_{d,D} (\widehat{\a}) := \Lambda_D^{n-1} (\sqrt{\lambda_D})^{\epsilon(\a)}\, \left(\Jtrs \circ \eta \right) (\a)\ ,
$$
where $\Lambda_D, \lambda_D$ are as defined in \eqref{Dlambda} and \eqref{CapitalLambda}, $\eta$ is as defined in \eqref{eta} and
$\epsilon(\a)$ is the sum of the exponents of the generators $\s_i$ and $\tau_i$ in the word  $\a$.
 Then the map  
 $$\Psi_{d,D}(q,z): {\mathcal L}_{\mathcal{S}} \rightarrow \mathcal{R}_D, \,\,L \mapsto \Psi_{d,D}(L)$$  is a $2$-variable isotopy invariant of oriented singular links.
\end{thm}
Moreover, in the image $\eta ({\mathcal S} B_n)$, we have 
\begin{equation}\label{gipi}
 g_i - g_i^{-1} = (q - q^{-1})e_i \quad \text{for all $i=1,\ldots,n-1$,}
\end{equation}
which gives rise to the following skein relation (compare with \cite{jula4}):
\begin{equation*}
\frac{1}{\sqrt{\lambda_D}}\, \Psi_{d,D} (L_{+})   -  \sqrt{\lambda_D}\, \Psi_{d,D} (L_{-}) 
= \frac{q - q^ {-1}}{\sqrt{\lambda_D}}\, \Psi_{d,D} (L_{\times})
\end{equation*}
where  $L_{+}$, $L_{-}$ and $L_{\times}$ are diagrams of three oriented singular links which are identical except for one crossing, where they are 
as in Figure~\ref{skeinsfig}.

\smallbreak
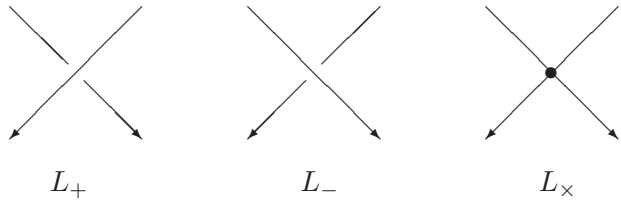
\begin{figure}[H]
\begin{center}
\label{skeinsfig}
\begin{picture}(230,80)
\put(50,80){\vector(-1,-1){50}}
\put(28,52){\vector(1,-1){22}}
\put(0,80){\line(1,-1){22}}

\put(118,58){\line(1,1){22}}
\put(112,52){\vector(-1,-1){22}}
\put(90,80){\vector(1,-1){50}}

\put(230,80){\vector(-1,-1){50}}
\put(180,80){\vector(1,-1){50}}

\put(202,52){$\bullet$}
\put(15,10){$L_+$}
\put(110,10){$L_{-}$}
\put(200,10){$L_{\times}$}
\end{picture}
\caption{The singular links $L_{+}$, $L_{-}$ and $L_{\times}$.}
\end{center}
\end{figure}

\begin{rem} \rm
The behaviour of the traces $\Jtrs$ under inversion, split links, connected sums and mirror imaging carries through to the invariants $\Psi_{d,D}$.
\end{rem}

\begin{rem} \rm The invariants defined in \cite{jula2,jula3,jula4} for framed, classical and singular links respectively, were derived from the old presentation of the Yokonuma--Hecke algebras $\YH(u)$ . These invariants are not necessarily topologically equivalent to the invariants $\Phi_{d,D}$, $\Theta_{d,D}$ and $\Psi_{d,D}$ defined above. See further \cite{chjukala}.
\end{rem}

\subsection{Transverse links}\label{trnsv}

Another class of links which is naturally related to the Yokonuma--Hecke algebras is the class of {\it transverse links}.
A transverse knot is represented by a smooth closed spacial curve which is nowhere tangent to planes of a special field of planes in $\R^3$ called {\it standard contact structure} (for the precise definition, see \cite{FT}).
Transverse links are naturally  framed  and oriented. Transverse equivalence of links consists of isotopy in the class of transverse links.  More precisely, a topological type of framed links is called {\it transversely simple} if any two transverse links within the type are transverse equivalent to each other. Not every topological type is simple; it may consist of several different transverse types. So, two links that are classically isotopic may be transversely non--equivalent (see examples in the next section). 
The problem is to find transverse invariants for such links.

In 1983, D.~Bennequin \cite{Be} noted that the closed braid presentation of knots is convenient for describing transverse knots with the blackboard framing.  For a link $L$ represented as a closed braid $\widehat{\a}$ with $n$ strands, one can check that the self-linking number is equal to $sl(L)=\epsilon(\a)-n$, where $\epsilon(\a)$  is the sum of the exponents of the braiding generators $\s_i$ in the word  $\a \in B_n$ \cite{Be}. So, a transverse knot $K$ represented by the closure of a braid $\a \in B_n$ defines naturally an element of the framed braid group $\a':= t_1^{sl(K)} \a \in {\mathcal F}_{n}$. Equivalently, we can write $\a'= t_1^{r_1}\ldots t_n^{r_n} \a$, where $r_1 + \cdots + r_n = sl(K)$. This generalizes to transverse links in the obvious manner (adding the exponents that correspond to a specific component gives the self-linking of that component).

Further, S.~Orevkov and V.~Shevchishin \cite{OrSh}  and independently N.~Wrinkle \cite{Wr} gave  a transverse analogue of the Markov Theorem,
 comprising conjugation in the braid groups and {\it only positive} stabilizations and destabilizations: $ \alpha \sim \alpha \,{\sigma}_n $, where $\alpha \in B_n$.

 Now, rule (3) of the definition of the trace $\Jtr$ (Theorem \ref{trace}) tells us that $\Jtr$ respects positive stabilizations. Let $L$ be a transverse link represented by the closure $\widehat{\a}$ of a braid $\a \in B_n$, giving rise to the framed braid $\a'= t_1^{r_1}\ldots t_n^{r_n} \a \in {\mathcal F}_n$.  We define 
$$
M_d(\widehat{\alpha}) :=  \frac{1}{z ^{n-1}} \Jtr(\a').
$$

\begin{thm}\label{transverseinv} Let ${\mathcal L}_{\mathcal{T}}$ denote the set of transverse links.
The map $$M_d(q,z, x_1, \ldots , x_{d-1}) : {\mathcal L}_{\mathcal{T}} \rightarrow \C[z^{\pm 1}], \,\,L \mapsto M_d(L)$$
is a $(d+1)$-variable isotopy invariant of oriented transverse links.
\end{thm}

\begin{proof}
By construction, $M_d$ is  invariant under conjugation. We will show that $M_d$ is invariant under positive stabilization and destabilization, that is, $M_d(\widehat{\a}) = M_d(\widehat{\a\s_n})$ for any $\a \in B_n$, $n \in \N$. We have:
$$
M_d(\widehat{\a\s_n}) = \frac{1}{z^n} \Jtr(\a \s_n) = \frac{1}{z^n} \, z \,\Jtr(\widehat{\a}) = M_d(\widehat{\a}).
$$
\end{proof}

\begin{rem}\label{rem:re-scal} \rm 
Due to the absence of the negative stabilization in the definition of $M_d$, the links $\widehat{\a \s_n}$ and $\widehat{\a \s_n^{-1}}$ need not take the same value of the invariants. Hence, the re-scaling map $\s_i \mapsto \sqrt{\lambda_D} g_i$ is not needed any more and by consequence a quantity analogous to $\lambda_D$ is not introduced. However, if we make such a rescaling and specialize 
$(x_1, \ldots , x_{d-1})$ to the solution  
$X_D=({\rm x}_1, \ldots , {\rm x}_{d-1})$ of the {\rm E}--system, then   the corresponding invariant of transverse links would coincide with the invariants $\Phi_{d,D} (q,z)$ of oriented framed links from Theorem \ref{framedinv}.
 
\end{rem}

\begin{rem} \rm
The behaviour of the traces $\Jtr$ under inversion and split links carries through to the invariants $M_d$.
\end{rem}

\section{Enabling computations for the Yokonuma--Hecke algebras}\label{comp}
Due to the complicated quadratic relations of the Yokonuma--Hecke algebras, calculations of interesting examples are very difficult without the use of the computer. M.~Chmutov
created a program in Maple and K.~Karvounis created a program in $C\#$  for computing the trace $\Jtr$ and the invariants $\Phi_{d,D}$ (and consequently also the invariants $\Theta_{d,D}$).
The two programs were compared and they are available on:
\begin{center}
\url{http://www.math.ntua.gr/~sofia/yokonuma}
\end{center}

\subsection{Description of the algorithm}

The algorithm used in both programs is based on recursive applications of the quadratic relation~\eqref{quadr} of the Yokonuma--Hecke algebras. 
Assume that we want to calculate the trace $\Jtr$ of an element $\alpha \in \mathcal{F}_{n}$.
First, all negative exponents of the braiding generators are eliminated through the inverse relation~\eqref{invrs}. This produces a $\C$-linear combination of new framed braid words with only positive exponents. Applying successively the quadratic relation on these words leads to a $\C$-linear combination of words not containing any braiding generator with exponent other than $1$. Then, by successive applications of the framing relations $(\mathrm{f}_1)$ and $(\mathrm{f}_2)$, all these words are reduced to their split form. If more squares of the braiding generators appear in the process, we apply again the quadratic relation. We repeat the above procedure until $\alpha$ is expressed as a   $\C$-linear combination of framed braid words (equivalently, algebra monomials in $\YH(q)$) which are in their  split form and do not contain any braiding generator with exponent other than $1$.

Next, by applying the braid relations $(\mathrm{b}_1)$ and $(\mathrm{b}_2)$, especially $(\mathrm{b}_1)$ and its more general form 
\begin{equation*}
\s_{j} \s_{j-1} \ldots \s_{i+1} \s_{i} \s_{i+1} \ldots \s_{j-1} \s_{j} = \s_{i} \s_{i+1} \ldots  \s_{j-1} \s_{j}  \s_{j-1} \ldots \s_{i+1} \s_{i}
\qquad \mbox{\text{for} $i<j$,}
\end{equation*}
and breaking any squares that appear in the process with the use of the quadratic relation, 
$\alpha$ is eventually expressed as a $\C$-linear combination of algebra monomials in $\YH(q)$ of the form:
$$t_1^{k_1}\ldots t_n^{k_n} g_{i_1} \ldots g_{i_r},$$
where $k_1,\ldots,k_n \in \Z$, $i_p \neq i_{p+1}$ and $g_{{\rm max}\{i_1,\dots,i_r\}}$ appears only once.
Set $I:= {\rm max}\{i_1,\dots,i_r\}$.
For all $j=1,\ldots,n$, let $\overline{k}_j$ denote $k_j \,({\rm mod}\, d)$. Set 
$J:= {\rm max}\{ 1 \leq j \leq n \,|\, \overline{k}_j \neq 0\}$.
 Note that we have 
 $$\Jtr (t_1^{k_1}\ldots t_n^{k_n} g_{i_1} \ldots g_{i_r})=\Jtr (t_1^{k_1}\ldots t_J^{k_J} g_{i_1} \ldots g_{i_r}).$$
 Then one of the following three cases occurs:
\begin{itemize}
\item $J > I+1$. In this case, rule (4) of $\Jtr$ can be applied: 
\begin{center}
$ \Jtr (t_1^{k_1}\ldots t_J^{k_J} g_{i_1} \ldots g_{i_r})=
x_{\overline{k}_{J}}\, \Jtr (t_1^{k_1} \ldots t_{J -1}^{k_{J -1}} g_{i_1} \ldots g_{i_r}).$
\end{center} \smallbreak
\item $J < I+1$. In this case, rule (3) of $\Jtr$ can be applied: 
\begin{center}
$\Jtr (t_1^{k_1}\ldots t_{J}^{k_{J}} g_{i_1} \ldots  g_I \ldots g_{i_r}) = z \,  \Jtr (t_1^{k_1}\ldots t_{J}^{k_{J}} g_{i_1} \ldots  \widehat{g_I}  \ldots g_{i_r}),$
\end{center}
where $\widehat{g_I}$ denotes the removal of $g_I$. \smallbreak
\item $J = I+1$. Then $t_{J}=t_{I+1}$ commutes with every braiding generator except  for $g_{I}$. We have:
\begin{center}
$\begin{array}{rcl}
\Jtr (t_1^{k_1}\ldots t_{I+1}^{k_{I+1}} g_{i_1} \ldots  g_I \ldots g_{i_r})  &= & \Jtr (t_1^{k_1}\ldots t_{I}^{k_{I}} g_{i_1} \ldots  g_I t_{I}^{k_{I+1}}  \ldots g_{i_r}) \\
 & = & z \, \Jtr (t_1^{k_1}\ldots t_{I}^{k_{I}} g_{i_1} \ldots  \widehat{g_I} t_{I}^{k_{I+1}}  \ldots g_{i_r}),
 \end{array}$
\end{center}
where $\widehat{g_I}$ denotes the removal of $g_I$.
\end{itemize}
Iterations of the above steps allow us to compute $\Jtr (t_1^{k_1}\ldots t_n^{k_n} g_{i_1} \ldots g_{i_r})$ and, eventually, through the linearity of the trace, $\Jtr(\a)$. 

The above algorithm has exponential complexity with respect to $d$ due to the recursive application of the quadratic relation; each application of the quadratic relation produces a linear sum of 
$d+1$ new words.

\subsection{The case of classical links}
If we restrict to the subset of classical links, it has been conjectured by J. Juyumaya that the application of the trace $\Jtrs$ on a classical braid word will not result in any appearance of the algebra generators $t_j$, except through the elements $e_i$. Thus, the last rule of the trace $\Jtrs$ can be replace by two other rules involving the elements $e_i$. This result has been indeed proved in \cite{chjukala}. Consequently, a computer program similar to the above has been developed in \cite{ka} for computing the invariants $\Theta_{d,D}$ treating the $e_i$'s as formal elements; hence in every application of the quadratic relation we obtain only two new words, which reduces greatly the computational complexity (see \cite{ka} for details). This program is also available on \url{http://www.math.ntua.gr/~sofia/yokonuma}. 

\section{Results on transverse knots}\label{results_transverse}

\subsection{Computational results} Our original hope was that the isotopy invariants $M_d$, defined in $\S \ref{trnsv}$, would distinguish the transverse knots of the same topological type and with the same Bennequin numbers \cite{Be}. However, our computer experiments show that the invariants  $M_d$  do not distinguish the simplest examples (which are so complicated that the direct computation by hand is impossible) from the infinite series introduced by Birman--Menasco \cite{BM}:
$$
\s_1^{2a+1}\s_2^{2b}\s_1^{2c}\s_2^{-1}\qquad\mbox{and}\qquad
\s_1^{2a+1}\s_2^{-1}\s_1^{2c}\s_2^{2b}\qquad
\mbox{with $a,b,c >1$ and $a+1\not=b\not=c$.}
$$
Then, using the inversion property (Proposition \ref{invProp}), we  managed to give a theoretical confirmation of our observation not only for the simplest examples but also for the two infinite series of transversely non-simple knots, the one above and the one by Khandhawit--Ng \cite{KN}:
$$
\s_3\s_2^{-2}\s_3^{2a+2}\s_2\s_3^{-1}\s_1^{-1}\s_2^2\s_1^{2b+1}
\qquad\mbox{and}\qquad
  \s_3\s_2^{-2}\s_3^{2a+2}\s_2\s_3^{-1}\s_1^{2b+1}\s_2^2\s_1^{-1}
\qquad\mbox{for $a,b\geqslant 0$,}
$$
starting with the knot $10_{132}$ with $a=b=0$.

We also computed some invariants  $M_d$
for the following 13 knots \cite{CN}, three of which ($m(9_{45})$, $10_{128}$, $10_{160}$) were only conjectured to be transversely non-simple:
$$m(7_2), m(7_6), 9_{44}, m(9_{45}), 9_{48}, 10_{128}, m(10_{132}),
10_{136}, m(10_{140}), m(10_{145}), 10_{160}, m(10_{161}), 12n_{591},
$$
where $m(K)$ stands for the mirror image of $K$.  The three knots of particular interest for being transverse non-equivalent have the following braid presentations \cite{Ng}:
$$
\begin{array}{rlc}
m(9_{45}): & \qquad\s_3^{-1}\s_2\s_1\s_3\s_2^{-1}\s_3\s_1\s_2^2
\qquad\mbox{and}\qquad
  \s_2^2\s_1\s_3\s_2^{-1}\s_3\s_1\s_2\s_3^{-1}\ ,& sl=1\ ;\\
10_{128}: & \qquad \s_1\s_2\s_1\s_2\s_1\s_2\s_1\s_3^2\s_2\s_3^{-1}
\qquad\mbox{and}\qquad
  \s_3^{-1}\s_2\s_3^2\s_1\s_2\s_1\s_2\s_1\s_2\s_1\ ,&sl=6\ ;\\
10_{160}: &\qquad\s_2^{-1}\s_3\s_2^{-1}\s_1^{-1}\s_3\s_2\s_3\s_2\s_3\s_1^2
\qquad\mbox{and}\qquad
  \s_1^2\s_3\s_2\s_3\s_2\s_3\s_1^{-1}\s_2^{-1}\s_3\s_2^{-1}\ ,&sl=1\ .\\
\end{array}
$$
Our computations were disappointing, and  the  explanation for this lies in the following.

\subsection{Trace invariants and Vassiliev invariants.}

Any quantum knot invariant can be expressed in terms of Vassiliev invariants in a standard way (see, for example, \cite{bi} or \cite{CDM}).
We show that the trace invariant $M_d$ can be similarly expressed in terms of Vassiliev invariants.

Let $v$ be a transverse knot invariant taking values in an abelian group. We can define its extension to singular knots with finitely many double points recursively, setting $v(L_{\times}):=v(L_{+})-v(L_{-})$ for the triple of knots in Figure \ref{skeinsfig}. An invariant $v$ is called a {\it Vassiliev invariant of order $\leq n$} if its extension vanishes on all singular knots with more than $n$ double points.

\begin{prop}
Let us make a substitution $q=e^h$ into the invariant  
$M_d(q, z, x_1, \ldots , x_{d-1})$  and consider the Taylor expansion in the power series in $h$. For every $n \in \N$, the coefficient of $h^n$ is a Vassiliev invariant of order $\leq n$.
\end{prop}

\begin{proof}
According to Equation \ref{gipi}, the difference of values of 
$M_d(e^h, z, x_1, \ldots , x_{d-1})$ on two knots $L_{+}$ and $L_{-}$ is divisible by 
$q-q^{-1} = e^h-e^{-h}= 2h + (\mbox{higher powers of}\ h)$.
Thus, the value of the extension on a knot with a singular point is divisible by $h$. Now, if there is another double point far away from the first one, then the value of the extension will be divisible by another $h$, so it will be divisible by $h^2$. Similarly, each additional double point contributes an extra $h$ in the result. Finally, for a knot with $n+1$ double points, the value of the invariant will be divisible by $h^{n+1}$. Therefore, all coefficients up to order $n$ will vanish on such a knot.
\end{proof}

The above proposition implies that the invariant 
$M_d(e^h, z, x_1, \ldots , x_{d-1})$ of transverse knots is covered by an (infinite) sequence of Vassiliev invariants.
However, the Fuchs--Tabachnikov theorem \cite[Theorem 5.6]{FT} claims that any transverse Vassiliev invariant turns out to be topological Vassiliev invariant of framed knots.

Similarly, with appropriate rescaling as in Remark \ref{rem:re-scal}
and specializing $(x_1, \ldots , x_{d-1})$ to the solution  
$X_D=({\rm x}_1, \ldots , {\rm x}_{d-1})$ of the {\rm E}--system, we can express the invariant $\Psi_{d,D}$ of singular links from Theorem \ref{Jula4:Th5} as an infinite sequence of Vassiliev invariants. Thus, because of the Fuchs--Tabachnikov theorem, $\Psi_{d,D}$ does not distinguish singular transverse knots of the same topological type and with the same Bennequin numbers.

\section{Comparing the invariants $\Theta_{d,D}$ for classical knots and links\\ with the Homflypt polynomial}\label{conjclassic}
This section contains a conjecture about the invariants $\Theta_{d,D}$ on \textit{knots} and computational evidence supporting this conjecture.

\subsection{Computations and supporting data} 

We computed the Yokonuma--Hecke  invariants $\Theta_{d,D}$ for various solutions of the {\rm E}--system for about 200 pairs of non-equivalent alternating knots and non-alternating knots and links having the same Homflypt polynomial.

Among all knots with up to $12$ crossings, there are 283 pairs or groups of knots whose members share the same Homflypt polynomial or one member and the mirror image of another. Their list, compiled by using the program {\it LinKnot} \cite{jasa}, is available in Table~\ref{knot_table} in Appendix~\ref{tables}.

Among all links with up to $12$ crossings there are $130$ pairs or
groups of links whose members (or one member with the mirror image of the other) have the same Homflypt polynomial. Since there is no universally accepted notation for links with more than $11$ crossings, we give in Table~\ref{link_table} bekiw only the list of links with up to $11$ crossings. The
complete list of links with up to $12$ crossings can be compiled by using
the program {\it LinKnot} \cite{jasa}, given in Conway notation,
braid notation, or by their DT codes.
\begin{table}[H]
\centering
\begin{tabular}{|c|c|c|} \hline
$9_{50}^2,L10n42$ & $9_{34}^2,L11n186$ & $L11a13,L11a323$ \\ \hline
$L10n79,L10n95$ & $L11n320,L11n329$ & $L11a149,L11a195$ \\ \hline
$L11a184,L11a207$ & $L11a172,L11a377$ & $L11a173,L11a382$ \\ \hline
\end{tabular}
\caption{Pairs of links with the same Homflypt polynomial}
\label{link_table}
\end{table}

\noindent As we shall state next, the invariants $\Theta_{d,D}$ agree on all knot pairs or groups sharing the same Homflypt polynomial. From our computations it is not possible to draw a definite conclusion about the behaviour of the invariants $\Theta_{d,D}$ on all mentioned link pairs of Table~\ref{link_table}.

\subsection{The invariants $\Theta_{d,D}$ for classical knots and the Homflypt polynomial}
While trying to compare computationally the invariants $\Theta_{d,D}(q,z)$  with the Homflypt polynomial $P(q,z)$, we noticed that the values of the invariants on \textit{knots} (links with only one component) were connected in the following way:
\begin{thm} \label{conjHomflypt}
Given a solution $X_D$ of the {\rm E}--system, for any braid $\a \in B_{n}$ such that $\widehat{\a}$ is a knot, we have 
$$\Theta_{d,D}(q,z) (\widehat{\a}) = \Theta_{1,\{0\}}  (q,{z}/{E_D}) (\widehat{\a}) = P(q,{z}/{E_D}) (\widehat{\a}).$$
\end{thm}

We have confirmed Theorem~\ref{conjHomflypt} on several knots from the table of knots as well as on knots from specific families from Table~\ref{families} below. Theorem~\ref{conjHomflypt} is proved in the sequel paper \cite{chjukala}.

\begin{rem} \rm
The transformation $z  \mapsto {z}/{E_D}$ corresponds to the transformation $\lambda_{\rm H} \mapsto \lambda_D$ for the Homflypt polynomial at variables $(q,\lambda_{\rm H})$. Hence, we have equivalently that $\Theta_{d,D}(q,\lambda_D) = P(q,\lambda_D)$, where the variable $\lambda_{\rm H}$ of the Homflypt polynomial has been substituted with the variable $\lambda_D$. Consequently, due to Theorem~\ref{conjHomflypt} and the above equality, $E_D$ does not appear in the values of the invariants $\Theta_{d,D}$ at variables $(q,\lambda_D)$ when computed on knots.
\end{rem}

\subsection{The invariants $\Theta_{d,D}$ for classical links}\label{links}
It does not seem, however, that a similar result holds for links. For Theorem~\ref{conjHomflypt} to hold for links, it must be that $\Jtrs(\a) = E_D^{n-1} \tau(\a)$, for some braid $\a$. Although, for example, for the link $\widehat{\s_1^{2k}}$ we have:
$$\Jtrs(\s_1^{2k}) = 1 -E_D + \frac{q^{2k}-q^{-2k}}{q+q^{-1}} z + \left( \frac{q^{2k-1}+q^{-(2k-1)}}{q+q^{-1}}  \right) E_D,$$ \smallbreak
$$\tau(\s_i^{2k}) = \frac{q^{2k}-q^{-2k}}{q+q^{-1}} \frac{z}{E_D} + \left( \frac{q^{2k-1}+q^{-(2k-1)}}{q+q^{-1}} \right) ,$$
where we consider the Ocneanu trace $\tau$ at variable $z/E_D$. Hence:
$$
\Jtrs(\s_1^{2k}) = 1 - E_D + E_D \, \tau(\s_i^{2k}),
$$
and by consequence:
$$
\Theta_{d,D}(q,z)(\s_1^{2k}) \neq P(q,z/E_D)(\s_1^{2k}).
$$

Further, we also computed the invariants $\Theta_{d,D}$ on families of links from Table~\ref{families} below and we arrived at the same conclusion. Namely, there was no visible pattern relating the invariants $\Theta_{d,D}$ with the Homflypt polynomial.
\begin{table}[H]
\renewcommand{\arraystretch}{1.2}
\begin{tabular}{|l|c|c|}
\hline 
Braids & Knots & Links \\ 
\hline 
$\s_1^p$ & $p \equiv 1 \mod 2$ & $p \equiv 0 \mod 2 $\\ 
$\s_1^p \s_2^{-1} \s_1^{-1} \s_2^{-1}$ & $p \equiv 1 \mod 2$ & $p \equiv 0 \mod 2$\\
$\s_1^p \s_2^{-1} \s_1 \s_2^{-1}$ & $p \equiv 1 \mod 2$ & $p \equiv 0 \mod 2$\\
$(\s_1 \s_2^{-1})^{3p}$ & -- & $\forall p$\\
$(\s_1^3 \s_2^{-1}) (\s_1 \s_2^{-1})^{3p-1}$ & -- & $\forall p$\\
$\s_1^{2p} \s_2 \s_1^{-1} \s_2$ & -- & $\forall p$\\
$\s_1^{2p} \s_2 \s_1^2 \s_2$ & -- & $\forall p$\\
$(\s_1^3 \s_2^{-1})^{2} \s_1 \s_2^{-1}$ & -- & (one link)\\
\hline
\end{tabular}
\caption{Families of knots and links}
\label{families}
\end{table}

In the sequel paper \cite{chjukala}, we provide a concrete formula for relating the invariants $\Theta_{d,D}$ on links with the Homflypt polynomial. Further, we are able to show that the invariants $\Theta_{d,D}$ distinguish six pairs of Homflypt-equivalent links (one of them is the fourth in Table~\ref{link_table}) and we give a diagrammatic proof for one pair.

\newpage
\appendix
\section{Table of Homflypt equivalent knots}\label{tables}
\tiny
\begin{table}[H]
\centering
\begin{tabular}{|c|c|c|c|} \hline
$10_{132},5_1$ & $6_2,K12n25$ & $K11n157,K12n482$ & $7_3,K12n523$\\ \hline
$7_1,K12n749$ & $K11n132,K11n50,K12n414$ & $K11n127,K11n22$ & $K12n10,K12n7$\\ \hline
$K11n128,K12n142$ & $10_{156},8_{16}$ & $K11a138,K11a285$ & $10_{33},K12n278$\\ \hline
$K12n367,K12n745$ & $K12n145,K12n838$ & $K12n400,K12n450$ &  $10_{105},K12n17,K12n584$\\ \hline
$K11n138,K11n79$ & $10_{23},K12n607$ & $10_{129},8_8$ & $K11n20,K12n322$\\ \hline
$10_{95},K12n595$ & $K11n17,K12n146$ & $9_8,K12n65$ & $K12n335,K12n710$\\ \hline
$K11a104,K11a168$ & $K12n421,K12n422$ & $K12n180,K12n633$ & $K12n429,K12n520$\\ \hline
$K11n63,K12n564$ & $10_{39},K12n215$ & $10_{93},K12n277$ & $K11n35,K11n43$\\ \hline
$K12n222,K12n58$ & $K11n71,K11n75$ & $K12n130,K12n85$ & $10_{155},K11n37$\\ \hline
$K12n123,K12n128$ & $8_2,K12n340$ & $K11n10,K11n144$ & $K12n126,K12n99$\\ \hline
$K11n40,K11n46$ & $K11n11,K11n112$ & $K11n56,K11n58,K12n449$ & $K11n51,K12n218,K12n54$\\ \hline
$10_{141},K12n438$ & $K11a282,K11a81$ & $K11n172,K12n312,K12n397$ & $10_{116},K12n424$\\ \hline
$K12n553,K12n556$ & $K11n55,K12n315$ & $K12n20,K12n634$ & $K11n21,K11n4,K12n24$\\ \hline
$K12n14,K12n298$ & $K11n26,K12n97$ & $10_{18},K12n561$ & $K12n420,K12n636$\\ \hline
$10_{88},K12n586$ & $K12n300,K12n463$ & $10_{38},K12n343$ & $10_{28},K12n342$\\ \hline
$K12n201,K12n551$ & $K12a154,K12a162$ & $K12a397,K12a83$ & $10_{83},K12n317,K12n511,K12n588$\\ \hline
$K12n606,K12n685$ & $K12a401,K12a81$ & $K12n209,K12n213$ & $K11a185,K11a265$\\ \hline
$K12n320,K12n534$ & $10_{92},K12n515$ & $K12n133,K12n88$ & $10_{27},K12n471,K12n645$\\ \hline
$K12n197,K12n687$ & $K11a183,K11a41$ & $10_{117},K12n415$ & $K12n540,K12n727$\\ \hline
$K12n560,K12n880$ & $K11a11,K11a167$ & $K12n486,K12n718$ & $K12n489,K12n746$\\ \hline
$K11a192,K11a299$ & $10_{84},K12n330$ & $K12a102,K12a107$ & $K12a707,K12a935$\\ \hline
$K12a108,K12a120$ & $K12a134,K12a188$ & $K12n417,K12n694$ & $K12n136,K12n91$\\ \hline
$K12a330,K12a619$ & $10_{35},K12n48$ & $K12a403,K12a625$ & $K12a140,K12a585$\\ \hline
$K11a272,K11a30$ & $K12a180,K12a560$ & $K12a338,K12a783$ & $K12a307,K12a781$\\ \hline
$K12a28,K12a799$ & $K12n223,K12n55$ & $K12n21,K12n29$ & $K12n224,K12n62,K12n66$\\ \hline
$K12n252,K12n262$ & $K12n255,K12n263$ & $K12n161,K12n36$ & $K12n144,K12n507$\\ \hline
$K12n288,K12n501$ & $K12n131,K12n86$ & $K12n26,K12n32$ & $K11a248,K11a71$\\ \hline
$K12n348,K12n70$ & $K12n3,K12n69$ & $K12n434,K12n616$ & $K12a514,K12a849$\\ \hline
$K12a513,K12a989$ & $K12a232,K12a322$ & $K12a275,K12a624$ & $K12a458,K12a887$\\ \hline
$K12a310,K12a388$ & $K12a705,K12a893$ & $K12a765,K12a961$ & $K12a510,K12a821$\\ \hline
$K12a1086,K12a945$ & $K11n73,K11n74$ & $K12n225,K12n63$ & $K11n151,K11n152$\\ \hline
$K12a1023,K12a926$ & $K12n132,K12n87$ & $K11n39,K11n45,K12n257$ & $K12a1071,K12a1175,K12a734$ \\ \hline
$K12n155,K12n677$ & $K11n34,K11n42$ & $K12n221,K12n56,K12n57$ & $K12n256,K12n264$\\ \hline
$K12n208,K12n212$ & $10_{52},K12n241$ & $K12n231,K12n232$ & $K12n210,K12n214$\\ \hline
$K12n364,K12n365$ & $10_{54},K12n154$ & $10_{57},K12n678$ & $10_{22},K12n817$\\ \hline
$K12n319,K12n470$ & $K12n194,K12n734$ & $K12a1173,K12a546$ & $K12n378,K12n819$\\ \hline
$10_{103},10_{40},K12n412$ & $K12n28,K12n34$ & $K12n380,K12n440$ & $K12n124,K12n129$\\ \hline
$K12n512,K12n80$ & $K12n23,K12n31$ & $10_{111},K12n112$ & $K12n159,K12n469$\\ \hline
$K12n565,K12n737$ & $K11a1,K11a149$ & $10_{25},10_{56}$ & $10_{86},K12n541$\\ \hline
$K12n532,K12n760$ & $10_{102},K12n447$ & $K12n147,K12n413,K12n597$ & $K12n339,K12n741$\\ \hline
$K12a1270,K12a588$ & $K12n117,K12n796$ & $K12a114,K12a117$ & $K12a44,K12a64$\\ \hline
$K11a116,K11a2$ & $10_{72},K12n94$ & $K12a385,K12a751$ & $K12a1172,K12a1268$\\ \hline
$K12n604,K12n666$ & $K11a186,K11a241$ & $K12a327,K12a677$ & $K12n344,K12n684$\\ \hline
$K12a111,K12a91$ & $K12n125,K12n98$ & $K12a778,K12a992$ & $K12a273,K12a890$\\ \hline
$K11a196,K11a216$ & $10_{106},K12n369$ & $K12a215,K12a280$ & $K12a352,K12a59,K12a63$\\ \hline
$K11a252,K11a254$ & $K12a13,K12a15$ & $K11a160,K11a289,K11a76$ & $10_{100},K12n188$\\ \hline
$K11a251,K11a253$ & $K12a164,K12a166$ & $K12n113,K12n466$ & $K11n41,K11n47,K12n326$\\ \hline
$K12a284,K12a673$ & $K12a226,K12a638$ & $K12a325,K12a711$ & $K12n618,K12n776$\\ \hline
$K12a328,K12a630$ & $K12a1076,K12a209$ & $K12a184,K12a333,K12a662$ & $K12a48,K12a60$\\ \hline
$K12a606,K12a715$ & $K12a658,K12a729$ & $K12a1032,K12a761$ & $K12n693,K12n696$\\ \hline
$K11a255,K11a79$ & $10_94,K12n787$ & $K12a14,K12a417,K12a7$ & $K11n36,K11n44$\\ \hline
$K11a33,K11a82$ & $K12a323,K12a452$ & $K11a24,K11a26,K11a315$ & $K12a157,K12a30,K12a33$\\ \hline
$K12a416,K12a71$ & $K12a1096,K12a402$ & $K12n135,K12n416,K12n90$ & $K12n138,K12n93$\\ \hline
$K11a19,K11a25$ & $10_{109},K12n695$ & $K11a316,K11a35$ & $K12a101,K12a115$\\ \hline
$K11a231,K11a57$ & $K12a136,K12a67$ & $K12n173,K12n27,K12n33$ & $K12n134,K12n89$\\ \hline
$K12n207,K12n228$ & $K12n670,K12n681$ & $K12n229,K12n67$ & $K12n137,K12n92$\\ \hline
$K12a113,K12a29$ & $K11n76,K11n78$ & $K12n671,K12n682$ & $K12n220,K12n59$\\ \hline
$K11a44,K11a47$ & $K12n691,K12n692$ & $K12n244,K12n338$ & $K12a811,K12a817$\\ \hline
$K12a240,K12a671$ & $K12a701,K12a987$ & $K12a527,K12a958$ & $K12a1083,K12a646$\\ \hline
$K12a599,K12a956$ & $K12a587,K12a977$ & $K12a343,K12a5$ & $K12a124,K12a384$\\ \hline
$K12a109,K12a437$ & $K12a1094,K12a306$ & $K12a212,K12a882$ & $K12a210,K12a623$\\ \hline
$K12a727,K12a947$ & $K12a1087,K12a456$ & $K12a1104,K12a478$ & $K12a844,K12a846$\\ \hline
$K12n22,K12n30$ & $K12n122,K12n127$ & $K12a112,K12a305$ & $K12a131,K12a133,K12a324,K12a933$\\ \hline
$K12a173,K12a258$ & $K12a24,K12a299$ & $K12a427,K12a435,K12a990$ & $K12a126,K12a132,K12a341,K12a627$\\ \hline
$K12a1013,K12a1176$ & $K12a1226,K12a916$ & $K12a1185,K12a656,K12a908$ & $K12a1047,K12a1203,K12a1246$\\ \hline
$K12n206,K12n227$ & $K12n205,K12n226$ & $K12n261,K12n64$ & $K12a45,K12a65$\\ \hline
$K12a195,K12a693$ & $K12a167,K12a692$ & $K12n219,K12n60,K12n61$ & $K12a116,K12a122,K12a182$\\ \hline
$K12a36,K12a694$ & $K12a639,K12a680$ & $K12a1136,K12a1224$ & $K12a675,K12a688$\\ \hline
$K12a1116,K12a41$ & $K12a1174,K12a837$ & $K12a830,K12a831$ & $K12a829,K12a832$\\ \hline
$K12a511,K12a988$ & $K12a274,K12a533$ & $K12a390,K12a672$ & $K12a526,K12a598,K12a954$\\ \hline
$K12a746,K12a966$ & $K12a1201,K12a523$ & $K12a1052,K12a1182$ &\\ \hline
\end{tabular}
\caption{Groups of knots up to 12 crossings with the same Homflypt polynomial}
\label{knot_table}
\end{table}

\end{document}